\documentclass[11pt,letterpaper]{article}

\usepackage{amsmath,amssymb,amsfonts,amsthm,mathrsfs}
\usepackage{graphicx}
\usepackage[usenames,dvipsnames]{color}
\usepackage[margin=1in]{geometry}
\usepackage{hyperref}

\theoremstyle{plain}
\newtheorem{thm}{Theorem}%[section]
\newtheorem{lem}[thm]{Lemma}
\newtheorem{cor}[thm]{Corollary}
\newtheorem{prop}[thm]{Proposition}
\newtheorem{remark}[thm]{Remark}
\newtheorem{conj}[thm]{Conjecture}
\newtheorem{que}[thm]{Question}
\theoremstyle{definition}

\theoremstyle{remark}
\newtheorem{claim}{Claim}

\newcommand{\real}{\ensuremath {\mathbb R} }
\newcommand{\ent}{\ensuremath {\mathbb Z} }
\newcommand{\nat}{\ensuremath {\mathbb N} }

\newcommand{\mbf}[1] {\text{\boldmath$#1$}}
\newcommand{\remove}[1] {}

\newcommand{\ex} {{\bf E}}
\newcommand{\pr} {{\bf Pr}}
\newcommand{\var} {{\bf Var}}

\newcommand{\B} {\ensuremath{\mathbb B}}
\newcommand{\M} {\ensuremath{\mathbb M}}

\newcommand{\cB} {\ensuremath{\mathcal B}}
\newcommand{\cC} {\ensuremath{\mathcal C}}
\newcommand{\bcG} {\ensuremath{\overline{\mathcal G}}}

\newcommand{\cF} {\ensuremath{\mathcal F}}
\newcommand{\cG} {\ensuremath{\mathcal G}}
\newcommand{\chG} {\ensuremath{\widehat{\mathcal G}}}
\newcommand{\cR} {\ensuremath{\mathcal R}}
\newcommand{\cT} {\ensuremath{\mathcal T}}
\newcommand{\cZ} {\ensuremath{\mathcal Z}}

\newcommand{\hX} {\ensuremath{\widehat X}}

\newcommand{\sG} {\ensuremath{\mathscr G}}
\newcommand{\sL} {\ensuremath{\mathscr L}}
\newcommand{\sM} {\ensuremath{\mathscr M}}
\newcommand{\bX} {\ensuremath{\mbf X}}
\newcommand{\bhX} {\ensuremath{\widehat{\mbf X}}}
\newcommand{\eps}{\varepsilon}

\newcommand{\nn}{\ensuremath{\widetilde n}}
\newcommand{\p}{{\widehat p}}
\newcommand{\pplus}{{\p\,^+}}
\newcommand{\pminus}{{\p\,^-}}

\DeclareMathOperator{\Bin}{Bin}
\DeclareMathOperator{\diam}{diam}

\title{Strong-majority bootstrap percolation on regular graphs with low dissemination threshold}
\author{Dieter Mitsche\thanks{e-mail: \texttt{dmitsche@unice.fr}, Universit\'{e} de Nice Sophia-Antipolis, Laboratoire J-A Dieudonn\'{e}, Parc Valrose, 06108 Nice cedex 02, France.} 
\and Xavier P\'erez-Gim\'enez\thanks{e-mail: \texttt{xperez@ryerson.ca}, Department of Mathematics, Ryerson University, Toronto, ON, Canada.} 
\and Pawe{\l} Pra{\l}at\thanks{e-mail: \texttt{pralat@ryerson.ca}, Department of Mathematics, Ryerson University, Toronto, ON, Canada; research partially supported by NSERC and Ryerson University.}}
\date{}

\begin{document}
\maketitle
\begin{abstract}
Consider the following model of strong-majority bootstrap percolation on a graph. Let $r \ge 1$ be some integer, and $p \in [0,1]$. Initially, every vertex is active with probability $p$, independently from all other vertices. Then, at every step of the process, each vertex $v$ of degree $\deg(v)$ becomes active if at least $(\deg(v)+r)/2$ of its neighbours are active. Given any arbitrarily small $p > 0$ and any integer $r$, we construct a family of $d=d(p,r)$-regular graphs such that with high probability all vertices become active in the end. In particular, the case $r=1$ answers a question and disproves a conjecture of Rapaport, Suchan, Todinca, and Verstra\"{e}te~\cite{RSTV}.
\end{abstract}

\section{Introduction}
Given a graph $G=(V,E)$, a set $A \subseteq V$, and $j \in \nat$, the \emph{bootstrap percolation} process $\B_j(G; A)$ is defined as follows: initially, a vertex $v \in V$ is \emph{active} if $v \in A$, and \emph{inactive} otherwise. Then, at each round, each inactive vertex becomes active if it has at least $j$ active neighbours.
The process keeps going until it reaches a stationary state in which every inactive vertex has less than $j$ active neighbours. We call this the \emph{final state} of the process.
Note that we may slow down the process by delaying the activation of some vertices, but the final state is invariant.
If $G$ is a $d$-regular graph, then there is a natural characterization of the final state in terms of the $k$-core (i.e.,~the largest subgraph of minimum degree at least $k$): the set of inactive vertices in the final state of $\B_j(G; A)$ is precisely the vertex set of the $(d-j+1)$-core of the subgraph of $G$ induced by the initial set of inactive vertices $V\setminus A$ (see~e.g.~\cite{JansonElPr}). We say that $\B_j(G; A)$ \emph{disseminates} if all vertices are active in the final state. 

Define $\B_j(G; p)$ to be the same bootstrap percolation process, where the set of initially active vertices is chosen at random: each $v \in V$ is initially active with probability $p$, independently from all other vertices.
This process (which can be regarded as a type of cellular automaton on graphs) was introduced in 1979 by Chalupa,  Leath and Reich~\cite{CLR79} on the grid $\ent^m$ as a simple model of dynamics of ferromagnetism, and has been widely studied ever since on many families of deterministic or random graphs.
The following obvious monotonicity properties hold: for any $A' \subseteq A'' \subseteq V$, if $\B_j(G; A')$ disseminates, then $\B_j(G; A'')$ disseminates as well; similarly, if $i \le j$ and $\B_j(G; A)$ disseminates, then $\B_i(G; A)$ must also disseminate. Therefore, the probability that $\B_j(G; p)$ disseminates is non-increasing in $j$ and non-decreasing in $p$.
In view of this, one may expect that, for some sequences  of graphs $G_n$,
there may be a sharp probability threshold $\p_n$ such that: for every constant $\eps>0$, a.a.s.\footnote{We say that a sequence of events $H_n$ holds \emph{asymptotically almost surely} (a.a.s.) if $\lim_{n\to\infty}\pr(H_n)=1$.} $\B_j(G_n; p_n)$ disseminates, if $p_n \ge (1+\eps) \p_n$; and a.a.s.\ it does not disseminate, if $p_n\le(1-\eps)\p_n$. If such a value $\p_n$ exists, we call it a \emph{dissemination threshold} of $\B_j(G_n; p_n)$.
Moreover, if $\lim_{n\to\infty} \p_n = \p \in[0,1]$ exists, we call this limit $\p$ the \emph{critical probability} for dissemination, which is \emph{non-trivial} if $0<\p<1$.
A lot of work has been done to establish dissemination thresholds or related properties of this process for different graph classes.
In particular, denoting $\{1,2,\ldots,n\}$ by $[n]$, the case of $G$ being the $m$-dimensional grid $[n]^m$ has been extensively studied: starting with the work of Holroyd~\cite{Holroyd} analyzing the $2$-dimensional grid, the results then culminated in~\cite{BBDM}, where Balogh et al.\ gave precise and sharp thresholds for the dissemination of $\B_j([n]^m; p)$ for any constant dimension $m\ge2$ and every $2\le j\le m$.
Other graph classes that have been studied are trees, hypercubes and hyperbolic lattices (see~e.g.~\cite{BPP, BB06, BBM, STBT}).

In the context of random graphs, Janson et al.~\cite{JLTV} considered the model $\B_j(G; A)$ with $j \geq 2$, $G=\sG(n,p)$\footnote{$\sG(n,p)$ is the probability space consisting of all graphs on $n$ vertices with vertex set $[n]$, and with each pair of vertices being connected by an edge with probability $p$, independently of all others.} and $A$ being a set of vertices chosen at random from all sets of size $a(n)$.
They showed a sharp threshold with respect to the parameter $a(n)$ that separates two regimes in which
the final set of active vertices has a.a.s.\  size $o(n)$ or $n-o(n)$ (i.e.~`almost' dissemination), respectively.
Moreover, there is full dissemination in the supercritical regime provided that $\sG(n,p)$ has minimum degree at least $j$. 
Balogh and Pittel~\cite{BP} analysed the bootstrap percolation process on random $d$-regular graphs, and established non-trivial critical probabilities for dissemination for all $2\le j\le d-1$, and Amini~\cite{Am10} considered random graphs with more general degree sequences. Finally, extensions to inhomogeneous random graphs were considered by Amini, Fountoulakis and Panagiotou in~\cite{AFP}.

Aside from its mathematical interest, bootstrap percolation was extensively studied by physicists: it was used to describe complex phenomena in jamming transitions~\cite{Toninelli}, magnetic systems~\cite{Sabha} and neuronal activity~\cite{Tlusty}, and also in the context of stochastic Ising models~\cite{Fontes}. For more applications of bootstrap percolation, see the survey~\cite{AlderLev} and the references therein.

\paragraph{Strong-majority model.}
In this paper, we introduce a natural variant of the bootstrap percolation process.
Given a graph $G=(V,E)$, an initially active set $A \subseteq V$, and $r\in\ent$, the {\em $r$-majority bootstrap percolation} process $\M_r(G;A)$ is defined as follows: starting with an initial set of active vertices $A$, at each round, each inactive vertex becomes active if the number of its active neighbours minus the number of its inactive neighbours is at least $r$. In other words, the activation rule for an inactive vertex $v$ of degree $\deg(v)$ is that $v$ has at least $\lceil(\deg(v)+r)/2 \rceil$ active neighbours.
As in ordinary bootstrap percolation, we are mainly interested in characterising the set of inactive vertices in the final state of and determining whether it is empty (i.e.~the process disseminates) or not.
Note that for a $d$-regular graph $G$, $\M_r(G;A)$ is exactly the same process as $\B_{\lceil (d+r)/2 \rceil}(G;A)$, and therefore the final set of inactive vertices of $\M_r(G; A)$ is precisely the vertex set of the
$\lfloor  (d - r)/2 + 1 \rfloor$-core of the graph induced by the initial set of inactive vertices. If $G$ is not regular, the two models are not comparable.
The process $\M_r(G;p)$ is defined analogously for a random initial set $A$ of active vertices, where each vertex belongs to $A$ (i.e.~is initially active) with probability $p$ and  independently of all other vertices.
Note that $\M_r(G;A)$ and $\M_r(G;p)$ satisfy the same monotonicity properties with respect to $A$, to $r$, and to $p$ that we described above for ordinary bootstrap percolation, and thus we define the \emph{critical probability} $\p$ for dissemination (if it exists) analogously as before. Additionally, for any (random or deterministic) sequence of graphs $G_n$, define
\begin{align*}
\pplus &= \inf\{p\in[0,1] : \text{a.a.s.\ $\M_r(G_n;p)$ disseminates}\}
\qquad\text{and}\\
\pminus &= \sup\{p\in[0,1] : \text{a.a.s.\ $\M_r(G_n;p)$ does not disseminate}\}.
\end{align*}
Trivially,  $0\le\pminus\le\pplus\le1$; and, in case of equality, the critical probability $\p$ must exist and satisfy $\p=\pminus=\pplus$.
The $r$-majority bootstrap percolation process is a generalisation of the \emph{non-strict majority} and \emph{strict majority} bootstrap percolation models, which correspond to the cases $r=0$ and $r=1$, respectively.
The study of these two particular cases has received a lot of attention recently.
For instance, Balogh, Bollob\'{a}s and Morris~\cite{BBM} obtained the critical probability $\p=1/2$ for the non-strict majority bootstrap percolation process $\M_0(G;p)$ on the hypercube $[2]^n$, and extended their results to the $m$-dimensional grid $[n]^m$ for $m \ge (\log\log n)^2 (\log\log\log n)$. Also, Stef\'ansson and Vallier~\cite{SV15} studied the non-strict majority model for the random graph $\sG(n,p)$. (Note that, since $\sG(n,p)$ is not a regular graph, this process cannot be formulated in terms of ordinary bootstrap percolation).
For the strict majority case, we first state a consequence of the work of Balogh and Pittel~\cite{BP} on random $d$-regular graphs mentioned earlier.
Let $\sG_{n,d}$ denote a graph chosen uniformly at random (u.a.r.\ for short) from the set of all $d$-regular graphs on $n$ vertices (note that $n$ is even if $d$ is odd).
Then, for any constant $d\ge3$, the critical probability for dissemination of the process $\M_1(\sG_{n,d}; p)$ is equal to
\begin{equation}\label{balogh}
\p(d) := 1 - \inf_{y \in (0,1)} \frac{y}{F(d-1,1-y)},
\end{equation}
where $F(d,y)$ is the probability of obtaining at most $d/2$ successes in $d$ independent trials with success probability equal to $y$. Moreover,
\begin{equation}\label{balogh2}
\p(3)=1/2,
\qquad
\min\{\p(d) : d\ge3\} = \p(7) \approx 0.269,
\qquad\text{and}\qquad
\lim_{d\to\infty}\p(d)=1/2.
\end{equation}
The case of strict majority was studied by Rapaport, Suchan, Todinca and Verstra\"{e}te~\cite{RSTV} for various families of graphs. 
They showed that, for the wheel graph $W_n$ (a cycle of length $n$ augmented with a single universal vertex), $\pplus$ is the unique solution in the interval $[0,1]$ to the equation $\pplus+(\pplus)^2-(\pplus)^3=\frac12$ (that is, $\pplus\approx0.4030$); and they also gave bounds on $\pplus$ for the toroidal grid augmented with a universal vertex.
Moreover, they proved that, for every sequence $G_n$ of $3$-regular graphs of increasing order (that is, $|V(G_n)| < |V(G_{n+1})|$ for all $n\in\nat$) and every $p<1/2$, a.a.s.\ the process $\M_1(G_n;p)$ does not disseminate (so $\pminus\ge1/2$). Together with the result from~\eqref{balogh2} that $\p(3)=1/2$, their result implies, roughly speaking, that, for every sequence of $3$-regular graphs, dissemination is at least as `hard' as for random $3$-regular graphs. In view of this, they conjectured the following:
\begin{conj}[\cite{RSTV}]\label{conjecture}
Fix any constant $d\ge3$, and let $G_n$ be any arbitrary sequence of $d$-regular graphs of increasing order.
Then, for the strict majority bootstrap percolation process on $G_n$, we have $\pminus \ge \p(d)$. That is, for any constant $0\le p< \p(d)$, a.a.s.\ the process $\M_1(G_n;p)$ does not disseminate.
\end{conj}
Observe that, if the conjecture were true, then for every sequence of $d$-regular graphs of growing order, $\pminus \ge \p(d) \ge \p(7) \approx 0.269$. This motivated the following question:
\begin{que}[\cite{RSTV}]\label{question}
Is there any sequence of graphs $G_n$ such that their critical probability of dissemination (for strict majority bootstrap percolation) is $\p=0$?
\end{que}
Further results for strict majority bootstrap percolation on augmented wheels were given in~\cite{IPLKiwi}, and some experimental results for augmented tori and augmented random regular graphs were presented in~\cite{RapaExp}. The underlying motivation in both papers (in view of Question~\ref{question}) was the attempt to construct sequences of graphs $G_n$ such that a.a.s.\ $\M_1(G_n;p)$ disseminates for small values of $p$ (i.e.,~sequences $G_n$ with a small value of $\pplus$). However, to the best of our knowledge, for all graph classes investigated before the present paper, the values of $\pplus$ obtained were strictly positive.
We disprove Conjecture~\ref{conjecture} by constructing a sequence of $d$-regular graphs such that $\pplus$ can be made arbitrarily small by choosing $d$ large enough (see Theorem~\ref{thm:main} and Corollaries~\ref{cor1} and~\ref{cor2} below). Moreover, by allowing $d\to\infty$, we achieve $\pplus=0$, and thus we answer Question~\ref{question} in the affirmative.
It is worth noting that, if one considers the non-strict majority model ($r=0$) instead of the strict majority model ($r=1$), then Question~\ref{question} has a trivial answer as a result of the work of Balogh et al.~\cite{BBDM} on the $m$-dimensional grid $[n]^m$. Indeed, their results imply that the process $\M_0([n]^m;p)$ has critical probability $\p=0$. (In fact, they establish a sharp threshold for dissemination at $\p(n)=\lambda/\log n \to 0$, for a certain constant $\lambda>0$). However, the aforementioned results do not extend to the strict majority model. As a matter of fact, it is easy to show that the process $\M_1([n]^m;p)$ has critical probability $\p=1$.
In order to prove this, observe that if all the vertices in the cube $\{1,2\}^m$ or any of its translates in the grid $[n]^m$ are initially inactive, then they remain inactive at the final state. If $p<1$, then each of these cubes is initially inactive with positive probability, so a.a.s.\ there exists an initially inactive cube and we do not get dissemination.

\paragraph{Our sequence of regular graphs.}
To state our results precisely, we first need to define a sequence of regular graphs that disseminates `easily'. For each $n\in\nat$ and $k=k(n)\in\nat$, consider the following graph $\sL(n,k)$: the vertices are the $n^2$ points of the toroidal grid $[n]^2$ with coordinates taken modulo $n$; each vertex $v=(x,y)$ is connected to the vertices $v+w$, where $w \in K :=\{-k,\ldots,-1,0,1,\ldots,k\} \times \{-1,1\}$. Assuming that $2k+1\le n$ (so that the neighbourhood of a vertex does not wrap around the torus), we have that $|K|=2(2k+1)=4k+2$, and thus our graph $\sL(n,k)$ is $(4k+2)$-regular. Therefore, in the process $\M_{2r}(\sL(n,k),p)$, an inactive vertex needs at least $2k+ r+1$ active neighbours to become active. Next, for even $n$ and $r=r(n)\in\nat$, we also consider the (random) graph $\sL^*(n,k,r)$, consisting of adding $r$ random perfect matchings to $\sL(n,k)$. These matchings are chosen u.a.r.\ from the set of perfect matchings of $[n]^2$ conditional upon not creating multiple edges (i.e.~the perfect matchings are pairwise disjoint and do not use any edge from $\sL(n,k)$). Note that $\sL^*(n,k,r)$ is $(4k+r+2)$-regular.
Moreover, the process $\M_r(\sL^*(n,k,r);p)$ has the same activation rule as $\M_{2r}(\sL(n,k);p)$: namely,
an inactive vertex becomes active at some round of the process if it has at least $2k+r+1$ active neighbours. In view of this and since $\sL(n,k)$ is a subgraph of $\sL^*(n,k,r)$, we can couple the two processes in a way that the set of active vertices of $\M_{2r}(\sL(n,k);p)$ is always a subset of that of $\M_r(\sL^*(n,k,r);p)$.
We will show that for every $p>0$ (and even $p =p(n) \to 0$ not too fast as $n\to \infty$) and every not too large $r\in\nat$, there is $k\in\nat$ such that a.a.s.\ $\M_{r}(\sL^*(n,k,r);p)$ disseminates. On a high level, our analysis comprises two phases: in phase $1$, we will consider 
$\M_{2r}(\sL(n,k);p)$ and show that most vertices become active in this phase. In phase $2$, we incorporate the effect of the $r$ perfect matchings and consider then $\M_r(\sL^*(n,k,r);p)$ to show that all remaining inactive vertices become active. This $2$-phase analysis is motivated by the fact that the final set of inactive vertices of $\M_r(\sL^*(n,k,r);p)$ is a subset of the final set of inactive vertices of $\M_{2r}(\sL(n,k);p)$, in view of the aforementioned coupling between the two processes.

\paragraph{Notation and results.} We use standard asymptotic notation for $n \to \infty$. All logarithms in this paper are natural logarithms. We make no attempt to optimize the constants involved in our claims.

\bigskip

Our main result is the following:
\begin{thm}\label{thm:main}
Let $p_0>0$ be a sufficiently small constant. Given any $p=p(n)\in[0,1]$, $k=k(n)\in\nat$ and $r=r(n)\in\nat$ satisfying (eventually for all large enough even $n\in\nat$),
\begin{equation}\label{eq:pkrthm}
200\frac{(\log\log n)^{2/3}}{(\log n)^{1/3}} \le p \le p_0, \qquad
\frac{1000}{p}\log(1/p) \le k \le \frac{p^2\log n}{3000\log(1/p)},
\quad\text{and}\quad
1\le r \le \frac{pk}{20},
\end{equation}
consider the $r$-majority bootstrap percolation process $\M_r(\sL^*(n,k,r);p)$ on the $(4k+r+2)$-regular graph $\sL^*(n,k,r)$, where each vertex is initially active with probability $p$. Then, $\M_r(\sL^*(n,k,r);p)$ disseminates a.a.s.
\end{thm}
\begin{remark}\hspace{0cm}
\begin{enumerate}
\item
By our assumptions on $p$, it is easy to verify that $\lceil \frac{1000}{p}\log(1/p) \rceil < \lfloor \frac{p^2\log n}{3000\log(1/p)} \rfloor$ (see~\eqref{eq:k0k1} in the proof of Proposition~\ref{prop:phase1}), and so the range for $k$ is non-empty, and the statement is not vacuously true. In particular, $k = \lceil \frac{1000}{p}\log(1/p) \rceil$ satisfies the assumptions of the theorem.

\item
Note that the lower bound required for $k$ in terms of $p$ is almost optimal: in Theorem~2 of~\cite{RSTV}, the authors showed (for the $1$-majority model) that for any sequence of $d$-regular graphs (of increasing order) with $d < 1/p$ (in the case of odd $d$) or $d < 2/p$ (in the case of even $d$), a.a.s.\ dissemination does not occur. (For the $r$-majority model with $r\ge 2$, dissemination is even harder.) Hence, setting $k = \lceil \frac{1000}{p}\log(1/p) \rceil$, our sequence of $\Theta(k)$-regular graphs $\sL^*(n,k,r)$ has the smallest possible degree up to an additional $\Theta(\log(1/p))$ factor for achieving dissemination.
\end{enumerate}
\end{remark}

As a consequence of Theorem~\ref{thm:main}, we get the following two corollaries. The first one follows from an immediate application of Theorem~\ref{thm:main} with
\[
p=200\tfrac{(\log\log n)^{2/3}}{(\log n)^{1/3}},
\qquad
k = \lfloor \tfrac{p^2\log n}{3000\log(1/p)} \rfloor
\qquad\text{and}\qquad
r =  \lfloor 400 \log\log n \rfloor,
\]
together with the monotonicity of the process $\M_r(\sL^*(n,k,r);p)$ with respect to $p$ and $r$.
\begin{cor}\label{cor2}
There is $d = \Theta \left( (\log n \cdot \log\log n)^{1/3} \right)$, and a sequence $G_n$ of $d$-regular graphs of increasing order such that, for every
\[
200\frac{(\log\log n)^{2/3}}{(\log n)^{1/3}} \le p \le 1
\qquad\text{and}\qquad
1 \le r \le  400 \log\log n,
\]
the process $\M_r(G_n;p)$ disseminates a.a.s.
\end{cor}
Setting $r=1$, this corollary answers Question~\ref{question} in the affirmative. The second corollary concerns the case in which all the parameters are constant.
\begin{cor}\label{cor1}

For any constants $0<p\le1$ and $r\in\nat$, there exists $d_0\in\nat$ satisfying the following. For every natural $d\ge d_0$, there is a sequence $G_n$ of $d$-regular graphs of increasing order such that the $r$-majority bootstrap percolation process
$\M_r(G_n;p)$ a.a.s.\ disseminates.
\end{cor}
\begin{proof}[Proof (assuming Theorem~\ref{thm:main})]
Fix $r\in\nat$. In view of the monotonicity of the process $\M_r(G_n;p)$ with respect to $p$, we only need to prove the statement for any sufficiently small constant $p>0$. In particular, we assume that $p\le p_0$ (where $p_0$ is the constant in the statement of Theorem~\ref{thm:main}) and also that $r+3 \le pk/20$, where $k_0=\lceil \frac{1000}{p}\log(1/p)\rceil$.
For any fixed natural $k\ge k_0$ and any $i\in\{0,1,2,3\}$, we apply Theorem~\ref{thm:main} with the same values of $p$ and $k$ but with $r+i$ instead of $r$. We conclude that there is a sequence $G_n$ of $d=(4k+r+2+i)$-regular graphs (of increasing order) such that $\M_{r+i}(G_n;p)$ disseminates a.a.s.\ (and thus $\M_r(G_n;p)$ also disseminates a.a.s., by monotonicity). Note that every natural
$d \ge 4k_0 + r +2$ was considered, and hence the proof of the corollary follows.
\end{proof}

In particular, since $\lim_{d\to\infty} \p(d)=1/2$ (cf.~\eqref{balogh2}),  Corollary~\ref{cor1} implies that, for every sufficiently large constant $d$, there is a sequence of $d$-regular graphs of increasing order such that (for the $1$-majority model) $\pplus < \p(d)$, which disproves Conjecture~\ref{conjecture}.

\paragraph{Organization of the paper.} In Section~\ref{sec:deterministic} we show that, given certain configurations, the set of active vertices of $\M_{r}(\sL(n,k);A)$ grows deterministically. Section~\ref{sec:perco} deals with Phase $1$ using tools from percolation theory. Section~\ref{sec:matchings} then analyzes the effect of the added perfect matchings, and concludes with the proof of the main theorem by combining the previous results with the right parameters. 
%
%-----------------------------------------------------------------------------------------
%
%
\section{Deterministic growth}\label{sec:deterministic}
In this section, we show that, under the right circumstances, the set of active vertices grows deterministically in $\M_{r}(\sL(n,k);A)$. For convenience, we will describe (sets of) vertices in $\sL(n,k)$ by giving their coordinates in $\ent^2$, and mapping them to the torus $[n]^2$ by the canonical projection. This projection is not injective, since any two points in $\ent^2$ whose coordinates are congruent modulo $n$ are mapped to the same vertex in $[n]^2$, but this will not pose any problems in the argument.

Given an integer $1\le m\le k$, we say a vertex $v$ is \emph{$m$-good} (or just \emph{good}) if each one of the following four sets contains at least $2\lceil k/m\rceil$ active vertices:
\[
v+\{1,2,\ldots,k\}\times\{1\}; \quad
v+\{1,2,\ldots,k\}\times\{-1\}; \quad
v-\{1,2,\ldots,k\}\times\{1\}; \quad
v-\{1,2,\ldots,k\}\times\{-1\}.
\]
Otherwise, call the vertex \emph{$m$-bad}.

For any nonnegative integers $a$ and $b$, we define the set $S^k_m(a,b)  \subseteq  [n]^2$ as 
\begin{equation*}\label{eq:cloud}
S^k_m(a,b) = \bigcup_{|i|\le m+a+1} [-x_i,x_i]\times \{i\},
\end{equation*}
where the sequence $x_i$ satisfies 
\begin{equation}\label{eq:x_i}
\begin{cases}
x_{m+a+1}=b
\\
x_i = x_{i+1} + k  & m\le i\le m+a
\\
x_i = x_{i+1} +  i\lceil k/m\rceil & 0\le i\le m-1
\\
x_{-i}=x_i & 0\le i\le m+a+1. 
\end{cases}
\end{equation}
(See Figure~\ref{fig:cloud} for a visual depiction of $S^k_m(a,b)$.)
\begin{figure}[h]
\centerline{\includegraphics{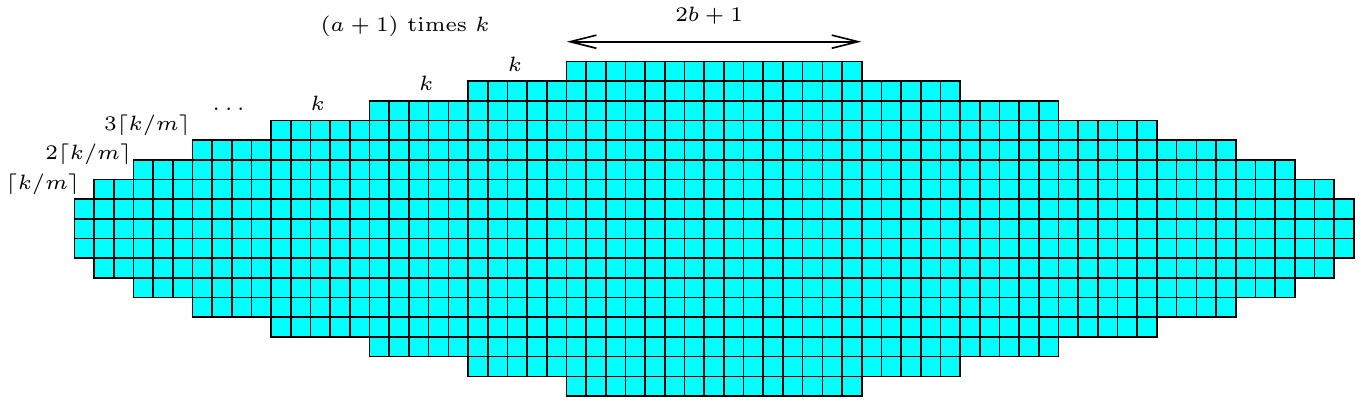}}
\caption{$S^k_m(a,b)$ with $m=5$, $k=5$, $a=2$ and $b=7$.}\label{fig:cloud}
\end{figure}
\remove{
\begin{figure}[h]
\centerline{\includegraphics{cloud2}}
\caption{$S^k_m(a,b)$ with large parameters.}\label{fig:cloud2}
\end{figure}
}
Observe that, since $k\ge m$, $\lceil k/m\rceil \leq 2k/m$, and therefore
\[
x_{0} =  b+(a+1)k + \sum_{i=1}^{m-1}i \lceil k/m\rceil  = b+(a+1)k + \lceil k/m\rceil \frac{m(m-1)}{2} \le b+(m+a)k,
\]
so 
\begin{equation}\label{eq:cloudin}
S^k_m(a,b) \subseteq [-b-(m+a)k,b+(m+a)k] \times [-m-a-1, m+a+1].
\end{equation}
In particular, 
\begin{equation}\label{eq:cloud0in}
S^k_m(0,0) \subseteq [-2mk,2mk] \times [-2m, 2m]
\qquad\text{and}\qquad
|S^k_m(0,0)| \le 25m^2k.
\end{equation}
Moreover, since $x_i \ge x_{i+1} + 1$ for $m\le i\le m+a$ (i.e.~the length of each row increases by at least one unit to the left and to the right) and a symmetric observation for rows $-m\le i\le -m-a$, we get
\begin{equation}\label{eq:cloudout}
S^k_m(2a,0) \supseteq [-a,a] \times [-a, a].
\end{equation}

A set of vertices $U\subseteq[n]^2$ is said to be \emph{active} if all its vertices are active. Note that $S^k_m(a,b) \subseteq S^k_m(a+1,b)$. The next lemma shows that, if $S^k_m(a,b)$ is active and all vertices in $S^k_m(a+1,b)$ are good (or already active), then eventually $S^k_m(a+1,b)$ becomes active too.
\begin{lem}\label{lem:growcloud}
 Given any integers $a,b\ge0$,  $1\le m < k$ and $r\le \lceil k/m\rceil$, suppose that $S^k_m(a,b)$ is active and all vertices in $S^k_m(a+1,b)$ are $m$-good or active in the $r$-majority bootstrap percolation process. Then, deterministically $S^k_m(a+1,b)$ eventually becomes active.
\end{lem}
\begin{proof}
Put $k'=\lceil k/m\rceil\ge2$. Note that any vertex with at least $2k+k'$ active neighbours has at most $2k+2-k'$ inactive neighbours, and thus becomes active since $(2k+k')-(2k+2-k') = 2(k'-1) \ge k'\ge r$.
Our first goal is to show that we can make active one extra vertex to the right and to the left of each row in $S^k_m(a,b)$. Let $x_i$ be as in~\eqref{eq:x_i}. For each $0\le i\le m+a+1$, 
consider the vertex $v_i = (x_i+1,i)$. Observe that $v_i\in S^k_m(a+1,b)$, so it must be active or good. If $v_i$ is active, then we are already done. Suppose otherwise that $v_i$ is good. By the definition of $S^k_ m(a,b)$, $v_i$ has at least $\min\{ k+(i-1)k', 2k\}$ neighbours in $S^k_m(a,b)$ one row below, and $\max\{k-ik',0\}$ one row above, so in particular at least $2k - k'$ neighbours in $S^k_m(a,b)$, which are active. Additionally, since $v_i$ is good, it has at least $2k'$  extra active neighbours above and to the right, so it becomes active. By symmetry, we conclude that, for every $|i|\le m+a+1$, vertices $(-x_i-1,i)$ and  $(x_i+1,i)$ become active. Therefore, all vertices in $S^k_m(a,b+1)$ become active.

A close inspection of~\eqref{eq:x_i} yields the following chain of inclusions:
\begin{equation}\label{eq:chain}
S^k_m(a,b)\subseteq S^k_m(a,b+1)\subseteq\cdots S^k_m(a,b+k)\subseteq S^k_m(a+1,b).
\end{equation}
In view of this,
the same argument can be inductively applied to show that for every $0 \leq j \leq k-1$, if all vertices in $S^k_m(a,b+j)$ are active, then we eventually reach a state in which all vertices in $S^k_m(a,b+j+1)$ become active as well. (Note that the argument requires that
the newly added vertices $v_i$ satisfy $v_i\in S^k_m(a+1,b)$, which follows from~\eqref{eq:chain}.)
 
 Finally, observe that all vertices in $[-b,b]\times\{-m-a-2, m+a+2\}$ have $2k+1$ neighbours in $S^k_m(a,b+k)$ (either in the row below or the row above). Since these vertices are good, they have at least $4k'$ active neighbours not in $S^k_m(a,b+k)$, and thus they become active too. We showed that all vertices in $S^k_m(a+1,b)$ became active, and the proof of the lemma is finished.
\end{proof}

We consider two other graphs $\sL_1(n)$ and $\sL_\infty(n)$ on the same vertex set $[n]^2$ as $\sL(n,k)$. Two vertices $(x,y)$ and $(x',y')$ in $[n]^2$ are adjacent in $\sL_1(n)$ if
\[
\begin{cases}
x'=x
\\
y'-y \equiv \pm 1 \mod n;
\end{cases}
\qquad\text{or}\qquad
\begin{cases}
y'=y
\\
x'-x \equiv \pm 1 \mod n.
\end{cases}
\]
Similarly, $(x,y)$ and $(x',y')$ are adjacent in $\sL_\infty(n)$ if
\[
(x,y)\ne(x',y')
\qquad\text{and}\qquad
\begin{cases}
x'-x \equiv 0,\pm 1 \mod n
\\
y'-y \equiv 0,\pm 1 \mod n.
\end{cases}
\]
In other words, $\sL_1(n)$ is the classical square lattice $n\times n$, and  $\sL_\infty(n)$ is the same lattice with diagonals added.
Given any two vertices $u,v\in[n]^2$, the $\ell_1$-distance and $\ell_\infty$-distance between $u$ and $v$ respectively denote their graph distance in $\sL_1(n)$ and  $\sL_\infty(n)$. (These correspond to the usual $\ell_1$- and $\ell_\infty$-distances on the torus.)
Also, we say that a set $U\subseteq[n]^2$ is \emph{$\ell_1$-connected} (or \emph{$\ell_\infty$-connected}) if the subgraph of $\sL_1(n)$ (or $\sL_\infty(n)$) induced by $U$ is a connected graph.
\remove{
We always measure distances between two elements in $[n]^2$ as the graph distance with respect to $\sL_1(n)$. (This is the usual $\ell_1$-distance on the torus.) Also, we say that a set $U\subseteq[n]^2$ is \emph{connected}, if the subgraph of $\sL_1(n)$ induced by $U$ is connected. Even when referring to vertices of $\sL(n,k)$, we use the notions of distance and connectivity we just described in terms of $\sL_1(n)$.
}
Given two sets $U,U'\subseteq [n]^2$, we say $U'$ is a translate of $U$ if there exists $(x,y)\in\ent^{2}$ such that $U'=(x,y)+U$ (recall that we interpret coordinates modulo $n$).

\begin{lem}\label{lem:growconn}
Let $k,m,r\in\ent$ satisfying $1\le m < k$ and $r \le \lceil k/m\rceil$. Suppose that $U\subseteq[n]^2$ has the following properties: $U$ is $\ell_1$-connected; all vertices in $[n]^2$ within $\ell_1$-distance at most $32mk^2$ from $U$ are $m$-good (or active); and $U$ contains an active set $S$ which is a translate of $S^k_m(0,0)$. Then, eventually $U$ becomes active in the $r$-majority bootstrap percolation process.
\end{lem}
\begin{proof}
Without loss of generality, we assume that $S=S^k_m(0,0)$ (by changing the coordinates appropriately). Then, by~\eqref{eq:cloud0in}, $S$ is contained inside the square $Q=[- 2mk, 2mk] \times [- 2mk, 2mk]$.
We weaken our hypothesis that $S\subseteq U$, and only assume that $Q\cap U\ne\emptyset$.
Let $S'=S^k_m(14mk,0)$.
% We have $S\subseteq S'$ and,
By~\eqref{eq:cloudin}, $S' \subseteq [-15mk^2,15mk^2]\times[-15mk^2,15mk^2]$. Therefore, every vertex in $S'$ must lie within $\ell_1$-distance $30mk^2+4mk \le 32mk^2$ from $U$, and thus must be good (or already active).
We repeatedly apply Lemma~\ref{lem:growcloud} and conclude that $S'$ eventually becomes active.
By~\eqref{eq:cloudout},  $S^k_m(14mk,0) \supseteq [-7mk,7mk]^2$, so $S'$ contains not only the square $Q$, but all $8$ translated copies of $Q$ around it. More precisely, for every $i,j\in\{-1,0,1\}$,
\[
S' \supseteq Q_{ij}, \qquad \text{where}\qquad Q_{ij} = (4mk+1)(i,j) +  Q.
\]
Hence, all nine squares $Q_{ij}$ eventually become active.

Note that, for any $x,y\in\ent$, the translate $\hat Q = (x,y)+Q$ contains $\hat S = (x,y)+S$.
Therefore,  if $\hat Q$ is active and intersects $U$, the argument above shows that all nine squares
\[
\hat Q_{ij} = (4mk+1)(i,j) +  \hat Q
\]
eventually become active as well.
We may iteratively repeat the same argument to any active translate of $Q$ that intersects $U$.
Since $U$ is $\ell_1$-connected, we can find a collection of translates of $Q$ that eventually become active and whose union contains $U$. This finishes the proof of the lemma.
\end{proof}
\subsection{The $t$-tessellation}\label{ssec:tessellation}
Given any integer $1\le t\le n$, we define the $t$-tessellation $\cT(n,t)$ of $[n^2]$ to be the partition of $[n]^2$ into cells
\[
C_{ij} = [a_i+1,a_{i+1}] \times [a_j+1,a_{j+1}], \qquad  0 \le i,j \le \lfloor n/t\rfloor - 1,
\]
where $a_i=it$ for $0\le i\le \lfloor n/t\rfloor - 1$ and $a_{\lfloor n/t\rfloor} = n$. Most cells in $\cT(n,t)$ are squares with $t$ vertices on each side, except for possibly those cells on the last row or column if $t\nmid n$. These exceptional cells are in general rectangles, and have between $t$ and $2t$ vertices on each side.

We may regard the set of cells $\cT(n,t)$ of the $t$-tessellation as the vertex set of either $\sL_1(\lfloor n/t\rfloor)$ or $\sL_\infty(\lfloor n/t\rfloor)$ (that is, $\cT(n,t) \simeq \big[ \lfloor n/t\rfloor \big]^2$) by identifying each cell $C_{ij}\in \cT(n,t)$ with $(i,j)\in \big[ \lfloor n/t\rfloor \big]^2$. Call each of the resulting graphs $\sL_1(n,t)$ and $\sL_\infty(n,t)$, respectively. In other words, the vertices of $\sL_1(n,t)$ are precisely the cells in $\cT(n,t) \simeq \big[ \lfloor n/t\rfloor \big]^2$, and each cell is adjacent to its neighbouring cells at the top, bottom, left and right (in a toroidal sense); and a similar description (adding the top-right, top-left, bottom-right and bottom-left cells to the neighbourhood) holds for $\sL_\infty(n,t)$. To avoid confusion, we always call the vertices of $\sL_1(n,t) \simeq \sL_1(\lfloor n/t\rfloor)$ and $\sL_\infty(n,t) \simeq \sL_\infty(\lfloor n/t\rfloor)$ cells, and reserve the word vertex for the original graph $\sL(n,k)$.

For $i\in\{1,\infty\}$, we say that a set of cells $\cZ\subseteq\cT(n,t)$ is \emph{$\ell_i$-connected}, if $\cZ$ induces a connected subgraph of $\sL_i(n,t)$. Also, the $\ell_i$-distance between two cells $C$ and $C'$ corresponds to their graph distance in the graph of cells $\sL_i(n,t)$. This should not be confused with the $\ell_i$-distance (in $\sL_i(n)$) between the vertices inside $C$ and $C'$. Sometimes, we will also refer to the $\ell_i$-distance between a vertex $v$ and a cell $C$. By this, we mean the minimum distance in $\sL_i(n)$ between $v$ and any vertex $u\in C$.

Given $1\le m\le k$, we say that a cell $C\in\cT(n,t)$ is \emph{$m$-good} (or simply  \emph{good}) if every vertex inside or within $\ell_1$-distance $32mk^2$ of $C$ is good or active. Otherwise, we call it \emph{bad}. Note that deciding whether a cell $C$ is good or bad only depends on the status of the vertices inside or within $\ell_1$-distance $32mk^2 + k+1$ from $C$.
We call a cell a \emph{seed} if it contains an active translate of $S^k_m(0,0)$. (By~\eqref{eq:cloud0in}, this definition is not vacuous if $t\ge4mk+1$.)

In view of all these definitions, Lemma~\ref{lem:growconn} directly implies the following corollary.
\begin{cor}\label{cor:growcells}
Let $k,m,r,t\in\ent$ satisfying $1\le m < k$,  $r\le \lceil k/m\rceil$ and $1\le t\le n$. Suppose that $\cZ$ is an $\ell_1$-connected set of cells in $\cT(n,t)$ such that all cells in $\cZ$ are $m$-good and $\cZ$ contains a seed.
Then, in the $r$-majority bootstrap percolation process, eventually all cells in $\cZ$ become active.
\end{cor}
%
%
%-----------------------------------------------------------------------------------------
%
%
\section{Percolative ingredients}\label{sec:perco}
In this section, we consider the $t$-tessellation $\cT(n,t)$ defined in Section~\ref{ssec:tessellation} for an appropriate choice of $t$. We combine the deterministic results in Section~\ref{sec:deterministic} together with some percolation techniques to conclude that eventually most cells in $\cT(n,t)$ (and thus most vertices in $\sL(n,k)$) will eventually become active a.a.s. This corresponds to Phase 1 described in the introduction.

Throughout the section, we define $\nn=\lfloor n/t\rfloor$ and assume that $\nn\to\infty$ as $n\to\infty$. We identify the set of cells $\cT(n,t)$ with $[\nn]^2$ in the terms described in Section~\ref{ssec:tessellation}, and consider the graphs of cells $\sL_1(n,t) \simeq \sL_1(\nn)$ and  $\sL_\infty(n,t) \simeq \sL_\infty(\nn)$. Recall (for $i\in\{1,\infty\}$) the definitions of $\ell_i$-connected sets of cells and $\ell_i$-distance between cells from that section. Moreover,  define an $\ell_i$-path of cells to be a path in the graph $\sL_i(\nn)$, and the $\ell_i$-diameter of an $\ell_i$-connected set of cells $\cZ$ to be the maximal $\ell_i$-distance between two cells $C,C'\in\cZ$. (The $\ell_i$-diameter of $\cZ$ is also denoted $\diam_{\ell_i}\cZ$.) Finally, given a set of cells $\cZ$, an $\ell_i$-component of $\cZ$ is a subset $\cC\subseteq\cZ$ that induces a connected component of the subgraph of $\sL_i(\nn)$ induced by $\cZ$.

We need one more definition to characterize very large sets of cells that ``spread almost everywhere'' in $[\nn]^2$.
Set $A=10^8$ hereafter. Given any $\varepsilon=\varepsilon(\nn) \in (0,1)$ and a set of cells $\cZ\subseteq[\nn]^2$, we say that $\cZ$ is \emph{$\varepsilon$-ubiquitous} if it satisfies the following properties:
\begin{itemize}
\item[(i)]
$\cZ$ is an $\ell_1$-connected set of cells;
\item[(ii)]
$|\cZ| \ge (1-A\varepsilon)\nn^2$; and
\item[(iii)]
given any collection $\cB_1,\cB_2,\ldots,\cB_j$ of disjoint $\ell_\infty$-connected non-empty subsets of $[\nn]^2\setminus \cZ$,
\begin{equation}\label{eq:diams}
\min_{1\le i\le j} \big\{ \diam_{\ell_\infty} \cB_i \big\} \le \frac{A}{\log(1/\varepsilon)}\log\left(\nn^2/j\right).
\end{equation}
\end{itemize}
In particular, (iii) implies that
\begin{itemize}
\item[(iv)]
every $\ell_\infty$-connected set of cells $\cB\subseteq [\nn]^2\setminus \cZ$ has $\ell_\infty$-diameter at most $\frac{A}{\log(1/\varepsilon)}\log(\nn^2)$.
\end{itemize}
Our goal for this section is to show that a.a.s.\ there is an $\varepsilon$-ubiquitous set of cells that eventually become active. As a first step towards this, we adapt some ideas from percolation theory to find an $\varepsilon$-ubiquitous set of good cells in $[\nn]^2$. We formulate this in terms of a slightly more general context. A \emph{$2$-dependent site-percolation model} on $\sL_1(\nn)$ is any probability space defined by the state (good or bad) of the cells in $[\nn]^2$ such that the state of each cell $C$ is independent from the state of all other cells at $\ell_1$-distance at least $3$ from $C$. We represent such probability space by means of the random vector $\bX=(X_C)_{C\in[\nn]^2}$, where $X_C$ is the indicator function of the event that a cell $C$ is good.
In this setting, let $\cG=\{C\in[\nn]^2 : X_C=1\}$ be the set of all good cells, and let $\cG_0$ be the largest $\ell_1$-component of $\cG$ (if $\cG$ has more than one $\ell_1$-component of maximal size, pick one by any fixed deterministic rule).
\begin{lem}\label{lem:perco}
Let $\epsilon_0>0$ be a sufficiently small constant. Given any $\epsilon=\epsilon(\nn)$ satisfying $ \nn^{-1/3}   <\epsilon\le \epsilon_0$, consider a $2$-dependent site-percolation model $\bX$ on $\sL_1(\nn)$, where each cell in $[\nn]^2$ is good with probability at least $1-\epsilon$.
Then, a.a.s.\ as $\nn\to\infty$, the largest $\ell_1$-component $\cG_0$ of the set of good cells is $\epsilon$-ubiquitous.
\end{lem}
\begin{proof}
Throughout the argument, we assume that $\epsilon_0$ is sufficiently small so that $\epsilon$ meets all the conditions required.
Let $\bcG_0 = [\nn]^2\setminus\cG_0$. Our first goal is to show the following claim.
\begin{claim} \label{claim:crossing}
A.a.s.\ every $\ell_\infty$-component of $\bcG_0$ has $\ell_\infty$-diameter at most $\nn/2$.
\end{claim}
For this purpose, we will use a classical result by Liggett, Schonmann, and Stacey (cf.~Theorem~0.0 in~\cite{LSS97}) that compares $\bX$ with the product measure. Given a constant $0<p_0<1$ (sufficiently close to $1$), consider $\bhX=(\hX_C)_{C\in[\nn]^2}$, in which the $\hX_C$ are independent indicator variables satisfying $\pr(\hX_C=1)=p_0$, and define $\chG = \{C\in[\nn]^2 : \hX_C=1\}$. If $\epsilon_0$ (and thus $\epsilon$) is small enough given $p_0$, then our 2-dependent site-percolation model $\bX$ stochastically dominates $\bhX$, that is, $\ex(F(\cG)) \ge \ex(F(\chG))$ for every non-decreasing function $F$ over the power set $2^{[\nn]^2}$ (i.e.~satisfying $F(\cZ)\le F(\cZ')$ for every $\cZ\subseteq\cZ'\subseteq[\nn]^2$).

Set $s = \lfloor\nn/4\rfloor$ and, for $i,j\in\{0,1,2,3,4\}$, consider the rectangles (in $\ent^2$)
\[
\cR_{i,j} = (is, js) + [1, s] \times [1, 2s]
\quad\text{and}\quad
\cR'_{i,j} = (is, js) + [1, 2s] \times [1, s].
\]
We regard $\cR_{i,j}$ and  $\cR'_{i,j}$ as subsets of the torus $[\nn]^2$ by interpreting their coordinates modulo $n$. Note that, if $4\mid \nn$ then some of these rectangles are repeated (e.g.~$\cR_{0,0}=\cR_{4,0}$), but this does not pose any problem for our argument. Let $\cR$ be any of the rectangles above and $\cZ\subseteq[\nn]^2$ be any set of cells. We say that $\cZ$ is $\ell_1$-crossing for $\cR$ if the set $\cZ\cap\cR$ has some $\ell_1$-component intersecting the four sides of $\cR$. It is easy to verify that if $\cZ$ is $\ell_1$-crossing for all $\cR_{i,j}$ and all $\cR'_{i,j}$, then every $\ell_\infty$-component of $[\nn]^2\setminus\cZ$ has $\ell_\infty$-diameter at most $2s \le \nn/2$. 
If $p_0$ is sufficiently close to $1$, by applying a result by Deuschel and Pisztora (cf.~Theorem~1.1 in~\cite{DP96}) to all $\cR_{i,j}$ and all $\cR'_{i,j}$, we conclude that a.a.s.\ $\chG$ contains an  $\ell_1$-component with more than $\nn^2/2$ cells which is $\ell_1$-crossing for all $\cR_{i,j}$ and all $\cR'_{i,j}$. This is a non-decreasing event, and hence a.a.s.\ $\cG$ has an $\ell_1$-component with exactly the same properties (which must be $\cG_0$ by its size). This implies the claim.

In view of Claim~\ref{claim:crossing}, we will restrict our focus to  $\ell_\infty$-components of $\bcG_0$ of small  $\ell_\infty$-diameter.
Let $N_d$ be the number of cells that belong to $\ell_\infty$-components of $\bcG_0$ of $\ell_\infty$-diameter $d$. Then, the following holds.
\begin{claim} \label{claim:Nmoments}
For every $0\le d\le \nn/2$,
\[
\ex N_d \le B \nn^2 \epsilon^{\lceil (d+ 1 )/4\rceil}
\quad (B=10^6)
\qquad\text{and}\qquad
\var N_d \le  (4d+ 5)^2\ex N_d.
\]
\end{claim}
In order to prove this claim, we need one definition.
A \emph{special sequence} of length $j$ is a sequence of $j+1$ different cells $C_0,C_1,\ldots,C_j$ in $[\nn]^2$ such that any two consecutive cells in the sequence are at $\ell_\infty$-distance exactly $3$, and any two different cells are at $\ell_\infty$-distance at least $3$. Observe that there are at most $24^j$ special sequences of length $j$ starting at a given cell $C_0$. Moreover, by construction, the states (good or bad) of the cells in a special sequence are mutually independent.

We now proceed to the proof of Claim~\ref{claim:Nmoments}. Let $\cB$ be an $\ell_\infty$-component of $\bcG_0$ of $\ell_\infty$-diameter $0\le d\le \nn/2$, and let $\cF$ be the set of cells inside $\cB$ but at $\ell_1$-distance $1$ of some cell in $\cG_0$. $\cF$ is $\ell_\infty$-connected (since $\sL_1(\nn)$ and $\sL_\infty(\nn)$ are dual lattices) and only contains bad cells. Moreover, $\cF$ must contain two cells $C$ and $C'$ at $\ell_\infty$-distance $d$ (with $C=C'$ if and only if $d=0$). Let $P=C_1,C_2,\ldots,C_m$ be a path joining $C=C_0$ and $C'=C_m$ in the subgraph of $\sL_\infty(\nn)$ induced by $\cF$. From this path, we construct a special sequence $Q=D_0,D_1,\ldots,D_{\lfloor d/3 \rfloor}$ as follows. Set $D_0=C_0$ and, for $1\le i\le \lfloor d/3 \rfloor$, $D_i=C_{j+1}$, where $C_j$ is the last cell in $P$ at $\ell_\infty$-distance at most $2$ from $D_{i-1}$. By construction, $Q$ is a special sequence of length $\lfloor d/3 \rfloor$ contained in $\cB$ and it consists of only bad cells.
Therefore, if any given cell $D\in[\nn]^2$ belongs to an $\ell_\infty$-component of $\bcG_0$ of $\ell_\infty$-diameter $d$, then there must be a special sequence of bad cells and length $\lfloor d/3 \rfloor$ starting within $\ell_\infty$-distance $d$ from $D$. This happens with probability at most
\[
(2d+1)^2 24^{\lfloor d/3 \rfloor} \epsilon^{1+\lfloor d/3 \rfloor} \le B \epsilon^{\lceil (d+1)/4 \rceil},
\]
where it is straightforward to verify that the last inequality holds for $B=10^6$ and all $d$, as long as $\epsilon_0$ is sufficiently small. Summing over all $\nn^2$ cells, we get the desired upper bound on $\ex N_d$.
To bound the variance, we consider separately pairs of cells that are within $\ell_\infty$-distance greater than $2d + 2$ and at most $2d + 2$, and we get
\[
\ex({N_d}^2) \le (\ex N_d)^2 + (4d+ 5)^2\ex N_d,
\]
so
\[
\var N_d \le  (4d+ 5)^2\ex N_d.
\]
This proves Claim~\ref{claim:Nmoments}. Next, let $N'_d=\sum_{i\ge d}N_i$ be the number of cells that belong to $\ell_\infty$-components of $\bcG_0$ of $\ell_\infty$-diameter at least $d$. Then, we have the next claim.
\begin{claim}\label{claim:Nprime}
A.a.s.\ for every $d\ge0$, $N'_d < B' \nn^2 \epsilon^{ \lceil (d+1)/5 \rceil }$, where $B'=11B$.
\end{claim}
Suppose first that $\ex N_d \ge \nn^{1/2}$. By Claim~\ref{claim:Nmoments}, we must have
$(1/\epsilon)^{\lceil (d+1)/4\rceil} \le B \nn^{3/2}$, so in particular $d \le \log \nn$. Then, using Chebyshev's inequality and the bounds in Claim~\ref{claim:Nmoments},
\begin{equation}\label{eq:ChebyNd}
\pr \left(N_d\ge2\ex N_d \right) \le \frac{\var N_d}{(\ex N_d)^2} \le \frac{(4d+5)^2}{\ex N_d} \le \frac{25\log^2\nn}{\nn^{1/2}}.
\end{equation}
Summing the probabilities over all $0\le d\le \log\nn$, the probability is still $o(1)$. Suppose otherwise that $\ex N_d \le  \nn^{1/2}$.
By Markov's inequality,
\begin{equation}\label{eq:MarkNd}
\pr\left(N_d \ge \nn^2\epsilon^{\lceil (d+1)/5\rceil}\right)
\le \frac{\ex N_d}{\nn^2 \epsilon^{\lceil (d+1)/5 \rceil}}.
\end{equation}
Recall from Claim~\ref{claim:Nmoments} and our assumptions that $\ex N_d \le \min\left\{ \nn^{1/2}, B \nn^2 \epsilon^{\lceil (d+1)/4\rceil} \right\}$.
If $\nn^{1/2} \le B \nn^2\epsilon^{\lceil (d+1)/4\rceil}$, then~\eqref{eq:MarkNd} becomes
\[
\pr\left(N_d \ge \nn^2\epsilon^{\lceil (d+1)/5\rceil}\right) \le \frac{1}{\nn^{3/2} \epsilon^{\lceil (d+1)/5 \rceil}}.
\]
For $0\le d \le 15$, the bound above is $o(1)$ as long as say $\epsilon \ge \nn^{-1/3}$.
For $d\ge16$, we have $\lceil (d+1)/5 \rceil + (d+1)/100 \le 0.95 \lceil (d+1)/4 \rceil$, and therefore
\[
\pr\left(N_d \ge \nn^2\epsilon^{\lceil (d+1)/5\rceil}\right)
\le \frac{1}{\nn^{3/2} \epsilon^{\lceil (d+1)/5 \rceil}}
\le \frac{\epsilon^{(d+1)/100}}{\nn^{3/2} \epsilon^{0.95\lceil (d+1)/4 \rceil}}
\le \frac{B^{0.95}  \epsilon^{(d+1)/100}}{\nn^{0.075 }},
\]
where for the last step we used that $(1/\epsilon)^{\lceil (d+1)/4\rceil} \le B \nn^{3/2}$. Summing the bound above over all $d\ge16$ gives again a contribution of $o(1)$.
Finally, if $B \nn^2\epsilon^{\lceil (d+1)/4\rceil} \le \nn^{1/2} $, then we must have $d\ge16$ since $\epsilon \ge \nn^{-1/3}$. Therefore~\eqref{eq:MarkNd} becomes
\[
\pr\left(N_d \ge \nn^2\epsilon^{\lceil (d+1)/5\rceil}\right)
\le \frac{B \nn^2\epsilon^{\lceil (d+1)/4\rceil}}{\nn^2 \epsilon^{\lceil (d+1)/5 \rceil}}
\le B \epsilon^{0.05\lceil (d+1)/4\rceil + (d+1)/100}
\le \frac{B^{0.95} \epsilon^{(d+1)/100}}{\nn^{0.075}},
\]
where for the last step we used that $\epsilon^{\lceil (d+1)/4\rceil} \le \nn^{-3/2}/B$. Summing the bound above over all $d\ge16$ gives $o(1)$.
Putting all the previous cases together, we conclude that a.a.s.\ for all $0\le d\le \nn/2$,
\[
N_d \le \max \left\{ \nn^2\epsilon^{\lceil (d+1)/5\rceil}, 2\ex N_d  \right\} \le 2B \nn^2\epsilon^{\lceil (d+1)/5\rceil}.
\]
The same is true for $d\ge\nn/2$ by Claim~\ref{claim:crossing}. Hence, a.a.s.\ for all $d\ge0$,
\[
N'_d = \sum_{i\ge d} N_i \le 2B \nn^2\epsilon^{\lceil (d+1)/5\rceil} \sum_{i\ge0} 5 \epsilon^i
< 11B \nn^2\epsilon^{\lceil (d+1)/5\rceil}.
 \]
This proves Claim~\ref{claim:Nprime}.

Finally, assume that the a.a.s.\ event in Claim~\ref{claim:Nprime} holds. Given any $1\le j\le \nn^2$, set
\[
d = \left\lfloor \frac{5 \log(B' \nn^2/j)}{\log(1/\epsilon)} \right\rfloor.
\]
Then, $\lceil(d+1)/5\rceil \ge \frac{\log(B' \nn^2/j)}{\log(1/\epsilon)}$, and so
\[
N'_d < B' \nn^2 \epsilon^{ \lceil(d+1)/5\rceil} \le  j.  
\]
Therefore, given any disjoint $\ell_\infty$-connected non-empty sets $\cB_1,\cB_2,\ldots,\cB_j \subseteq\bcG_0$ (not necessarily components), at least one of the $j$ sets must have $\ell_\infty$-diameter strictly less than $d$. Hence,
\[
\min_{1\le i\le j} \big\{ \diam_{\ell_\infty} \cB_i \big\} \le d-1
\le \frac{5 \log(\nn^2/j) + 5 \log B'}{\log(1/\epsilon)} - 1
\le \frac{5 \log(\nn^2/j)}{\log(1/\epsilon)}.
\]
This proves part~(iii) of the definition of $\epsilon$-ubiquitous for $\cG_0$. Part~(i) is immediate since $\cG_0$ is $\ell_1$-connected by definition. Finally, since $N'_0 < B' \nn^2 \epsilon$, then $|\cG_0| > \nn^2(1-B'\epsilon)$, which implies part~(ii).
So $\cG_0$ is $\epsilon$-ubiquitous.
\end{proof}
The next result combines Corollary~\ref{cor:growcells} and Lemma~\ref{lem:perco} in order to show that most of the cells become active during Phase~1 of the process.
\begin{prop}\label{prop:phase1}
Let $0<p_0<1$ be a sufficiently small constant. Given any $p=p(n)\in\real$, $k=k(n)\in\nat$ and $r=r(n)\in\ent$ satisfying
(eventually for all $n\in\nat$ sufficiently large)
\begin{equation}\label{eq:pkr}
200\frac{(\log\log n)^{2/3}}{\log^{1/3}n} \le p \le p_0, \qquad
\frac{1000}{p}\log(1/p) \le k \le \frac{p^2\log n}{3000\log(1/p)},
\qquad\text{and}\qquad
r \le pk/9,
\end{equation}
define
\begin{equation}\label{eq:kte}
t = t(n) = 100k^3
\qquad\text{and}\qquad
\varepsilon = \eps(n) = k^{-100}.
\end{equation}
Consider the $r$-majority bootstrap percolation process $\M_r(\sL(n,k); p)$, and the $t$-tessellation $\cT(n,t)$ of $[n]^2$ into $\nn^2 = \lfloor n/t\rfloor^2$ cells.
Then, a.a.s.\ the set of all cells that eventually become active contains an $\varepsilon$-ubiquitous $\ell_1$-component.
\end{prop}
\begin{proof}
Assume that $p_0$ is sufficiently small and $n$ sufficiently large so that the parameters $p$, $k$, $t$ and $\varepsilon$ satisfy all the required conditions below in the argument. (In particular, we may assume that $k,r,t$ are larger than a sufficently large constant, and $\eps$ is smaller than a sufficiently small constant.)
Define $k_0 = \left\lceil \frac{1000}{p}\log(1/p) \right\rceil$ and $k_1 = \left\lfloor \frac{p^2\log n}{3000\log(1/p)} \right\rfloor$. From~\eqref{eq:pkr} and since $p_0$ is small enough,
\begin{equation}\label{eq:k0k1}
k_0 < \frac{2000}{p}\log(1/p) = \frac{2000p^2\log^2(1/p)}{p^3\log(1/p)}
\le \frac{2000}{200^3 } \frac{p^2 \log n}{\log(1/p)}
< k_1,
\end{equation}
so there exist $k\in\nat$ satisfying $k_0\le k\le k_1$, and thus the statement is not vacuous.
Later in the argument we will need the bound
\begin{equation}\label{eq:pk8}
\frac{pk}{8} = \frac{pk}{8\log k} \log k \ge  \frac{pk_0}{8\log k_0} \log k \ge  \frac{900}{8} \log k \ge  111 \log k.
\end{equation}
Define $m=\lceil 8/p\rceil$, so in particular
\[
m < \frac{9}{p} < k_0 \le k,
\]
as required for the definition of $m$-good.
Moreover, $k \le k_1 < \log n < \frac{n-1}{2}$,
so every vertex of $\sL(n,k)$ has exactly $4k+2$ neighbours (i.e.~neighbourhoods in $\sL(n,k)$ do not wrap around the torus).
The number of vertices that are initially active in a set of $k$ vertices is distributed as the random variable $\Bin(k,p)$. Thus, by Chernoff bound (see, e.g., Theorem~4.5(2) in~\cite{MU05}), the probability that a vertex is initially $m$-bad is at most
\begin{equation}\label{eq:vbad}
4\pr\big(\Bin(k,p)< 2\lceil k/m\rceil\big) \le  4\pr\big(\Bin(k,p) \le (1-1/2 )pk\big) \le 4\exp(-pk/8),
\end{equation}
where we used that $2\lceil k/m\rceil \le 2\lceil pk/8 \rceil \le pk/2$.

Now consider the $t$-tessellation $\cT(n,t)$ of $[n]^2$ with $t = 100k^3$. In particular, we have
\begin{equation}\label{eq:tn}
t \le 100{k_1}^3 < \log^3 n <  n,
\end{equation}
so $\cT(n,t)$ is well defined.
For each cell $C\in\cT(n,t)$, let $X_C$ denote the indicator function of the event that $C$ is $m$-good.
Recall that every cell $C$ is a rectangle with at most $2t$ vertices per side, and thus $C$ has at most
$(2t+64mk^2)^2 \le 300^2 k^6$ vertices within $\ell_1$-distance $32mk^2$. Then, by~\eqref{eq:vbad}, \eqref{eq:pk8} and a union bound,
\[
\pr(X_C=0) \le 4(300^2 k^6)\exp(-pk/8)
 \le 600^2 k^6 \exp(-111\log k)
 \le (1/k)^{100} = \varepsilon.
 \]
Moreover, the outcome of $X_C$ is determined by the status (active or inactive) of all vertices within $\ell_1$-distance $32mk^2+k+1 \le 100mk^2 \le t$ from some vertex in $C$. All these vertices must belong to cells that are within $\ell_1$-distance at most $2$ from $C$ (recall that this refers to the distance in the graph of cells $\sL_1(n,t)$).
Therefore, for every cell $C\in\cT(n,t)$ and set of cells $\cZ\subseteq\cT(n,t)$ such that $C$ is at $\ell_1$-distance greater than $2$ from all cells in $\cZ$, the indicator $X_C$  is independent of $(X_{C'})_{C'\in\cZ}$.
Hence, $\bX=(X_C)_{C\in\cT(n,t)}$ is a $2$-dependent site-percolation model on the lattice $\sL_1(n,t)$ with $\pr(X_C=1)\ge 1-\varepsilon$.
Observe that $\bX$ satisfies the conditions of Lemma~\ref{lem:perco},
assuming that $\varepsilon=(1/k)^{100}$ is small enough (which follows from our choice of $p_0$) and since $\eps \ge {k_1}^{-100} > \log^{-100} n > \lfloor n/t\rfloor^{-1/3}$ (recall by~\eqref{eq:tn} that $t\le \log^3 n$, so the number of cells in $\cT(n,t)$ is $\nn^2 = \lfloor n/t\rfloor^2\to\infty$.) Then, by Lemma~\ref{lem:perco},
the largest $\ell_1$-component $\cG_0$ induced by the set of $m$-good cells is  a.a.s.\ $\varepsilon$-ubiquitous.
In particular
\begin{equation}\label{eq:PCmaxsmall}
\pr\left(|\cG_0|<(1-A\varepsilon)\lfloor n/t\rfloor^2\right) = o(1),
\end{equation}
where $A=10^8$.
We want to show that a.a.s.\ $\cG_0$ contains a seed.
For each cell $C\in\cT(n,t)$, let $Y_C$ be the indicator function of the event that
\[
S_C = (x+\lfloor t/2\rfloor,y+\lfloor t/2\rfloor) + S^k_m(0,0)
\]
is  initially active, where $(x,y)$ are the coordinates of the bottom left vertex in $C$. By~\eqref{eq:cloud0in}, $S_C$ is contained in $C$, and at $\ell_1$-distance greater than 
$\lfloor t/2\rfloor - 2mk > 40k^3 > 32mk^2+k+1$
from any other cell in $\cT(n,t)$, and therefore $Y_C$ depends only on vertices inside $C$ and at distance greater than $32mk^2+k+1$ from any other cell. In 
particular, $Y_C=1$ implies that $C$ is a seed.
Moreover, for any two disjoint sets of cells $\cZ,\cZ'\subseteq\cT(n,t)$, the random vectors
$(Y_C)_{C\in\cZ}$ and $(X_{C'})_{C'\in\cZ'}$ are independent, since they are determined by the status of two disjoint sets of vertices. For the same reason, $(Y_C)_{C\in\cZ}$ and $(Y_{C'})_{C'\in\cZ'}$ are also independent.
By~\eqref{eq:cloud0in} and~\eqref{eq:pkr}, the probability that a cell $C$ is a seed is at least
\begin{equation}\label{eq:PYC}
\pr(Y_C=1) \ge p^{25m^2k} \ge
p^{25(9/p)^2  (p^2\log n)/(3000\log(1/p))} =
e^{- (45^2/3000)  \log n} \ge
n^{-1}.
\end{equation}
For each cell $C$, define $\bar X_C = 1- X_C$ and $\bar Y_C = 1- Y_C$. Moreover, for each set of cells $\cZ$, let
\[
X_\cZ = \prod_{C\in\cZ} X_C,
\qquad
\bar X_\cZ = \prod_{C\in\cZ} \bar X_C,
\qquad
Y_\cZ = \prod_{C\in\cZ} Y_C
\qquad\text{and}\qquad
\bar Y_\cZ = \prod_{C\in\cZ} \bar Y_C.
\]
Now fix an $\ell_1$-connected set of cells $\cZ$ containing at least an $1-A\varepsilon$ fraction of the cells, and let $\partial\cZ$ be the set of cells not in $\cZ$ but adjacent in $\sL_1(n,t)$ to some cell in $\cZ$ (i.e.~the strict neighbourhood of $\cZ$ in $\sL_1(n,t)$).
Since $A\varepsilon<1/2$, the event $\cG_0=\cZ$ is the same as $X_\cZ \bar X_{\partial\cZ}=1$. Furthermore,
\begin{align*}
\pr\big((\bar Y_\cZ=1) \cap (X_\cZ \bar X_{\partial \cZ}=1)\big)
&= \pr\big((\bar Y_\cZ=1) \cap (\bar X_{\partial \cZ}=1)\big) - \pr\big((\bar Y_\cZ=1)\cap (X_\cZ=0)\cap (\bar X_{\partial \cZ}=1)\big)
\\
&\le \pr(\bar Y_\cZ =1) \pr(\bar X_{\partial \cZ}=1) - \pr(\bar Y_\cZ=1)  \pr\big((X_\cZ=0) \cap (\bar X_{\partial \cZ}=1)\big)
\\
&= \pr(\bar Y_\cZ=1)  \pr(X_\cZ \bar X_{\partial \cZ}=1),
\end{align*}
where we used that $\bar Y_\cZ$ and $\bar X_{\partial \cZ}$ are independent (since $\cZ$ and $\partial\cZ$ are disjoint sets of cells) and the fact that events $(\bar Y_\cZ=1)$ and $(X_\cZ=0) \cap (\bar X_{\partial \cZ}=1)$ are positively correlated (by the  FKG inequality --- see~e.g.~Theorem~(2.4) in~\cite{Gri99} --- since they are both decreasing properties with respect to the random set of active vertices). Therefore, using~\eqref{eq:PYC}, the independence of $Y_C$ and~\eqref{eq:tn}, we get
\begin{align*}
\pr\big(\bar Y_\cZ=1 \mid \cG_0=\cZ \big) &\le \pr\big(\bar Y_\cZ=1 \big) = \prod_{C\in\cZ} \pr\big(Y_C=0 \big)
\le \left(1-n^{-1}\right)^{|\cZ|}
\\
& \le \exp\left(-n^{-1} (1-A\varepsilon)\lfloor n/t\rfloor^2 \right) \le \exp\left(- (1-A\varepsilon) n^{-1+15/8} \right)=o(1).
\end{align*}
This bound is valid for all $\cZ$ with $|\cZ|\ge(1-A\varepsilon)\lfloor n/t\rfloor^2$, and hence
\[
\pr\big( (\text{$\cG_0$ has no seed}) \cap |\cG_0|\ge(1-A\varepsilon)\lfloor n/t\rfloor^2 \big) = o(1).
\]
Combining this with~\eqref{eq:PCmaxsmall}, we conclude that $\cG_0$ has a seed a.a.s.
When this is true, deterministically by Corollary~\ref{cor:growcells}, $\cG_0$ must eventually become active. Since we already proved that $\cG_0$ is a.a.s.\ $\varepsilon$-ubiquitous, the proof is completed.
\end{proof}
%
%
%------------------------------------------------------------------------------
%
%
\section{The perfect matchings}\label{sec:matchings}
In this section, we analyze the effect of adding $r$ extra perfect matchings to $\sL(n,k)$ regarding the strong-majority bootstrap percolation process, and prove Theorem~\ref{thm:main}.
Throughout this section we assume $n$ is even, and restrict the asymptotics to this case. An $r$-tuple $\sM=(\sM_1,\sM_2,\ldots,\sM_r)$ of perfect matchings of the vertices in $[n]^2$ is \emph{$k$-admissible} if $\sM_1\cup\sM_2\cup\cdots\cup\sM_r\cup\sL(n,k)$ (i.e.~the graph resulting from adding the edges of all $\sM_i$ to $\sL(n,k)$) does not have multiple edges. Observe that, if $1\le r\le n/2$, then such $k$-admissible $r$-tuples exist: for instance, given a cyclic permutation $\sigma$ of the elements in $[n/2]$, we can pick each $\sM_j$ to be the perfect matching that matches each vertex $(x,y)\in[n/2]\times[n]$ to vertex $(n/2+\sigma^{j-1}(x),y)$.
Note that $\sL^*(n,k)$ is precisely the uniform probability space of all possible graphs $\sM_1\cup\sM_2\cup\cdots\cup\sM_r\cup\sL(n,k)$ such that $\sM$ is a $k$-admissible $r$-tuple of perfect matchings of $[n]^2$.

The following lemma will be used to bound the probability of certain unlikely events for a random choice of a $k$-admissible $r$-tuple $\sM$ of perfect matchings of $[n]^2$.
\begin{lem}\label{lem:switchings}
Let $S\subseteq Z\subseteq[n]^2$ with $|S|=4s$ for some $s \ge 1$, $|Z|=z$, and suppose that  $z+2(4k+r+2)^2(4rs) \le n^2/2$ and $4erz \le n^2/2$.
Let $\sM=(\sM_1,\sM_2,\ldots,\sM_r)$ be a random $k$-admissible $r$-tuple of perfect matchings of $[n]^2$. The probability that every vertex in $S$ is matched by at least one matching in $\sM$ to one vertex in $Z$ is at most 
\[
(16rz/n^2)^{2s}.
\]
\end{lem}
\begin{proof}
Let $H_w$ be the event that there are exactly $w$ edges in $\sM_1\cup\sM_2\cup\cdots\cup\sM_r$ with one endpoint in $S$ and the other one in $Z$ (possibly also in $S$).
Note that the event in the statement implies that $\bigcup_{2s\le w\le 4rs}H_w$ holds. We will use the switching method to bound $\pr(H_w)$. For convenience, with a slight abuse of notation, the set of choices of $\sM$ that satisfy the event $H_w$ is also denoted by $H_w$.

Given any arbitrary element in $H_w$ (i.e.~given a fixed $k$-admissible $r$-tuple $\sM$ satisfying event $H_w$), we build an element in $H_0$ as follows. Let $u_1v_1,u_2v_2,\ldots,u_wv_w$ be the edges with one endpoint $u_i\in S$ and the other one $v_i\in Z$ (if both endpoints belong to $S$, assign the roles of $u_i$ and $v_i$ in any deterministic way), and let $1\le c_i\le r$ be such that $u_iv_i$ belongs to the matching $\sM_{c_i}$.
Let $R=\{u_1,\ldots,u_w,v_1,\ldots,v_w\}$. Throughout the proof, given any $U\subseteq[n]^2$, we denote by $N(U)$ the set of vertices that belong to $U$ or are adjacent in $\sM_1\cup\sM_2\cup\cdots\cup\sM_r\cup\sL(n,k)$ to some vertex in $U$.
Now we proceed to choose vertices $u'_1,u'_2,\ldots,u'_w$ and $v'_1,v'_2,\ldots,v'_w$ as follows.
Pick $u'_1\notin N(N(R))\cup Z$ and let $v'_1$ be the vertex adjacent to $u'_1$ in $\sM_{c_1}$; for each $1<i\le r$, pick $u'_i \notin N(N(R\cup\{u'_1,\ldots,u'_{i-1},v'_1,\ldots,v'_{i-1}\}))\cup Z$ and let $v'_i$ be the vertex adjacent to $u'_i$ in $\sM_{c_i}$. Since
\begin{eqnarray*}
| N(N(R\cup\{u'_1,\ldots,u'_{w},v'_1,\ldots,v'_{w}\}))\cup Z | &\le&  4w+4w(4k+r+2)+4w(4k+r+2)^2+z \\
&\le& 2(4k+r+2)^2(4rs)+z \le n^2/2,
\end{eqnarray*}
then there are at least
\[
(n^2/2)^w
\]
choices for $u'_1,u'_2,\ldots,u'_w$ ($v'_1,v'_2,\ldots,v'_w$ are then determined). We delete the edges $u_iv_i$ and  $u'_iv'_i$, and replace them by $u_iu'_i$ and  $v_iv'_i$. This switching operation does not create multiple edges, and thus generates an element of $H_0$.

Next, we bound from above the number of ways of reversing this operation. Given an element of $H_0$, there are exactly $4rs$ edges in $\sM_1\cup\sM_2\cup\cdots\cup\sM_r$ incident to vertices in $S$ (each such edge has exactly one endpoint in $S$ and one in $[n]^2\setminus Z$). We pick $w$ of these $4rs$ edges. Call them $u_1u'_1,u_2u'_2,\ldots,u_wu'_w$, where $u_i\in S$ and $u'_i\in[n]^2\setminus Z$, and let $1\le c_i\le r$ be such that $u_iu'_i\in\sM_{c_i}$. Pick also vertices $v_1,v_2,\ldots,v_w\in Z$, and let $v'_i$ be the vertex adjacent to $v_i$ in $\sM_{c_i}$. Delete $u_iu'_i$ and  $v_iv'_i$, and replace them by $u_iv_i$ and  $u'_iv'_i$. There are at most
\[
\binom{4rs}{w}z^w \le \left(\frac{4ersz}{w}\right)^w \le (2erz)^w
\]
ways of doing this correctly, and thus recovering an element of $H_w$. Therefore, $(n^2/2)^w|H_w| \le (2erz)^w |H_0|$, so $\pr(H_w) \le (4erz/n^2)^w$. Hence, we bound the probability of the event in the statement by
\[
\sum_{w=2s}^{4rs} \pr(H_w) \le \sum_{w\ge 2s} (4erz/n^2)^w \le (4erz/n^2) \sum_{w\ge 0} 2^{-w} = 2(4erz/n^2)^{2s} \le (16 rz/n^2)^{2s}.
\]
This proves the lemma.
\end{proof}

Given $1\le t\le n$, consider the $t$-tessellation $\cT(n,t)$ defined in Section~\ref{ssec:tessellation}. Recall that we identify the set of cells $\cT(n,t)$ with $[\nn]^2$, where $\nn=\lfloor n/t\rfloor$.
Given a $k$-admissible $r$-tuple $\sM$ of perfect matchings, we want to study the set of cells $\cR\subseteq[\nn]^2$ that contain vertices that remain inactive at the end of the process $\M_r(\sM_1\cup\sM_2\cup\cdots\cup\sM_r\cup\sL(n,k);p)$. The following lemma gives a deterministic necessary condition that ``small'' $\ell_\infty$-components of $\cR$ must satisfy, regardless of the initial set $U$ of inactive vertices. Recall that the set of vertices that remain inactive at the end of the process is precisely the vertex set of the $(2k+2)$-core of the subgraph induced by $U$.
\begin{lem}\label{lem:needstable}

Given any $r,k,t,n\in\nat$ (with even $n$) satisfying
\[
2r < 2k +2 \le t\le n/2,
\]
let $\sM$ be a $k$-admissible $r$-tuple of perfect matchings of the vertices in $[n]^2$, and let $U\subseteq[n]^2$ be any set of vertices. Let $U^\circ\subseteq U$ denote the vertex set of the $(2k+2)$-core of the subgraph of $\sM_1\cup\sM_2\cup\cdots\cup\sM_r\cup\sL(n,k)$ induced by $U$.
Assuming that $U^\circ\neq\emptyset$,
let $\cR$ be the set of all cells in the $t$-tessellation $\cT(n,t)$ that contain some vertex in $U^\circ$;
and let $\cB$ be an $\ell_\infty$-component of $\cR$ of diameter at most $\nn/2$ in $\sL_\infty(n,t)$.
Then, $\bigcup_{C\in\cB} C$ must contain at least $4$ vertices $v_1,v_2,v_3,v_4$ such that each $v_i$ is matched by some matching of $\sM$ to a vertex in $U^\circ$.
\end{lem}
\begin{proof}
We first include a few preliminary observations that will be needed in the argument.
Note that the condition $2k +2 \le t\le n/2$ implies that $\cT(n,t)$ has at least $2\times2$ cells, and also that the neighbourhood of any vertex in $\sL(n,k)$ has smaller horizontal (and vertical) length than the side of any cell in $\cT(n,t)$ (so that the neighbourhood does not cross any cell from side to side, and does not wrap around the torus).
Set $A=[n]^2\setminus U$ (we can think of $A$ and $U$ as the sets of initially active and inactive vertices, respectively),
%Any vertex in $U$ with at least $2k+r+1$ neighbours in $A$ with respect to the graph $\sM_1\cup\sM_2\cup\cdots\cup\sM_r\cup\sL(n,k)$ must have at most $2k+1$ neighbours in $U$, and thus does not belong to $U^\circ$.
and define $B=\bigcup_{C\in\cB}C$, namely the set of all vertices in cells in $\cB$. Any two vertices $v$ and $w$ that are adjacent in $\sL(n,k)$ must belong to cells at $\ell_\infty$-distance at most $1$ in $\cT(n,t)$. In particular,  if $v\in B$ and $w\notin B$, then $w$ must belong to some cell not in $\cR$ (since $\cB$ is an $\ell_\infty$-component of $\cR$), and therefore $w\in A$ (so $w\notin U^\circ$). % We will make use of this previous fact later.
Finally, since the $\ell_\infty$-diameter of $\cB$ is at most $\nn/2$,
 $B$ can be embedded into a rectangle that does not wrap around the torus $[n]^2$.
All geometric descriptions (such as `top', `bottom', `left' and `right') in this proof concerning vertices in $B$ should be interpreted with respect to this rectangle.

In view of all previous ingredients, we proceed to prove the lemma.
Let $v_{\mathtt T}$ (respectively, $v_{\mathtt B}$) be any vertex in the top row (respectively,  bottom row) of $B\cap U^\circ$, which is non-empty by assumption.
Suppose for the sake of contradiction that $v_{\mathtt T} = v_{\mathtt B}$. Then, $B\cap U^\circ$ has a single row, and the leftmost vertex $v$ of this row has no neighbours (with respect to the graph $\sL(n,k)$) in $U^\circ$. Indeed, from an earlier observation, any neighbour of $v$ lies either in $B$ (and thus in a row different from $B\cap U^\circ$) or in $A$ (and then not in $U^\circ$). Therefore, $v$ has at most $r<2k+2$ neighbours in $U^\circ$ with respect to the graph $\sM_1\cup\sM_2\cup\cdots\cup\sM_r\cup\sL(n,k)$, which contradicts the fact that $v\in U^\circ$.
We conclude that $v_{\mathtt T} \ne v_{\mathtt B}$.
Let $v_{\mathtt L}$ (respectively, $v_{\mathtt R}$) be the topmost vertex in the leftmost column (respectively, rightmost column) of $B\cap U^\circ$. Similarly as before, if $v_{\mathtt L} = v_{\mathtt T}$, then $v_{\mathtt L}$ has at most $k+1$ neighbours in $U^\circ$ with respect to $\sL(n,k)$ (the ones below and not to the left of $v_{\mathtt L}$), and thus
at most $r+k+1<2k+2$ neighbours in $U^\circ$ with respect to $\sM_1\cup\sM_2\cup\cdots\cup\sM_r\cup\sL(n,k)$, which leads again to contradiction. Therefore, $v_{\mathtt L} \neq v_{\mathtt T}$ and, by a symmetric argument, $v_{\mathtt L} \neq v_{\mathtt B}$, $v_{\mathtt R} \neq v_{\mathtt T}$ and $v_{\mathtt R} \neq v_{\mathtt B}$. This also implies $v_{\mathtt L} \neq v_{\mathtt R}$ (since otherwise, $v_{\mathtt L} = v_{\mathtt T} = v_{\mathtt R}$).
Hence, the vertices $v_{\mathtt T},v_{\mathtt B},v_{\mathtt L},v_{\mathtt R}$ are pairwise different, and each of them has at most $2k+1$ neighbours in $U^\circ$ with respect to the graph $\sL(n,k)$ (this follows again from the extremal position of $v_{\mathtt T},v_{\mathtt B},v_{\mathtt L},v_{\mathtt R}$ in $B\cap U^\circ$, together with the earlier fact that a neighbour of $v\in B$ not in $B$ must belong to $A$). Therefore, $v_{\mathtt T},v_{\mathtt B},v_{\mathtt L},v_{\mathtt R}$ must be matched by at least one matching in $\sM$ to other vertices in $U^\circ$.
\end{proof}
The conclusion of this lemma motivates the following definition. A collection of sets of cells $\cB_1,\cB_2,\ldots,\cB_s\subseteq\cT(n,t)$ is said to be \emph{stable} (w.r.t.\ a $k$-admissible $r$-tuple $\sM$ of  perfect matchings) if, for every set $\cB_j$, there are at least 4 vertices in $\bigcup_{C\in\cB_j} C$ that are matched by some perfect matching of $\sM$ to some vertex in $\bigcup_{i=1}^s \bigcup_{C\in\cB_i} C$. So the conclusion of Lemma~\ref{lem:needstable} says that the small $\ell_\infty$-components of $\cR$ must form a stable collection of sets with respect to $\sM$.
In Section~\ref{sec:perco}, we showed that, for an appropriate choice of parameters, the set of cells that are active at the end of Phase~1 is  a.a.s.\ contains an $\varepsilon$-ubiquitous $\ell_1$-component (recall that we apply Phase~1 to $\M_{2r}(\sL(n,k);p)$). % and resume the process in Phase~2 by regarding at $\M_r(\sL^*(n,k,r);p)$.  }
If this event occurs, then the set of cells that are active after Phase~2 (i.e.~after adding a $k$-admissible $r$-tuple $\sM$ of perfect matchings, and resuming the strong-majority bootstrap percolation process) must also contain an $\varepsilon$-ubiquitous $\ell_1$-component, deterministically regardless of the matchings. In particular, the set of cells $\cR$ containing some inactive vertices at the end of the process must contain at most $A\varepsilon\nn^2$ cells, and every subset of $\ell_\infty$-components of $\cR$ must satisfy~\eqref{eq:diams}. Moreover, by Lemma~\ref{lem:needstable}, the collection of $\ell_\infty$-components of $\cR$ must be stable with respect to $\sM$. The following lemma shows that for a randomly selected $k$-admissible perfect matching $\sM$, a.a.s.\ there is no proper set of cells $\cR$ satisfying all these properties. Therefore, assuming that Phase~1 terminated with an $\varepsilon$-ubiquitous set of active cells, Phase~2 ends with all cells (and thus all vertices) active a.a.s.
\begin{lem}\label{lem:nostable}
Let $0<\varepsilon_0 < 1/(2A)$ be a sufficiently small constant (where $A=10^8$). Given any $\varepsilon=\varepsilon(n)\in\real$, $k=k(n)\in\nat$, $r=r(n)\in\nat$ and $t=t(n)\in\nat$ satisfying (eventually for all large enough even $n\in\nat$)
\begin{equation}\label{eq:kret}
1\le r \le k, \quad
0<\varepsilon\le\varepsilon_0
\quad\text{and}\quad
1\le kt^5 \le \min \left\{ (1/\varepsilon)^{1/4}, n/\log^6 n \right\},
\end{equation}
consider the $t$-tessellation $\cT(n,t)$ of $[n]^2$, and pick a $k$-admissible $r$-tuple $\sM$ of  perfect matchings of the vertices in $[n]^2$ uniformly at random. Set $\nn=\lfloor n/t\rfloor \to \infty$. Then, the following holds a.a.s.: for any $1\le s\le A\varepsilon\nn^2$ and any collection of disjoint $\ell_\infty$-connected sets of cells $\cB_1,\cB_2,\ldots,\cB_s$ satisfying
\begin{equation}\label{eq:diams2}
\min_{1\le i\le j} \{ \diam_{\ell_\infty}(\cB_i) \} \le \frac{A}{\log(1/\varepsilon)}\log(\nn^2/j) \qquad \forall 1\le j\le s,
\end{equation}
the collection $\cB_1,\cB_2,\ldots,\cB_s$ is not stable with respect to $\sM$.
\end{lem}
\begin{proof}
We assume throughout the proof that $\varepsilon_0$ is sufficiently small and $n$ sufficiently large, so that all the required inequalities in the argument are valid. In particular, by~\eqref{eq:kret}, $k\le (n-1)/2$, so the neighbourhood with respect to $\sL(n,k)$ of any vertex does not wrap around the torus.

Given $1\le s\le A\varepsilon\nn^2$, suppose there exists a collection of $s$ pairwise-disjoint $\ell_\infty$-connected sets of cells $\{ \cB_1,\cB_2,\ldots,\cB_s \}$ satisfying~\eqref{eq:diams2} and which is stable with respect to $\sM$.
Assume w.l.o.g.\ that $\diam_{\ell_\infty}(\cB_1)\ge\cdots\ge\diam_{\ell_\infty}(\cB_s)$, so in particular
\[
\diam_{\ell_\infty}(\cB_i)\le d_i \qquad \forall i\in[s], \qquad\text{where}\quad  d_i=\frac{A}{\log(1/\varepsilon)}\log(\nn^2/i).
\]
This implies that there must exist $4s$ distinct vertices $v_{i,\ell}$ ($i\in[s]$, $\ell\in[4]$) with the following properties. Let $C_{i,\ell}$ be the cell containing $v_{i,\ell}$, and let $\cZ_i \supseteq \cB_i$ be the set of cells in $\cT(n,t)$ within $\ell_\infty$-distance $d_i$ from $C_{i,1}$. (Note that not necessarily $\cZ_i \cap \cZ_j = \emptyset$ for $i \ne j$.) Then, for each $i\in[s]$, the cells $C_{i,2},C_{i,3},C_{i,4}$ are within $\ell_\infty$-distance $d_i$ from $C_{i,1}$ (i.e.~$C_{i,2},C_{i,3},C_{i,4}\in\cZ_i$). Moreover, putting $\cZ=\bigcup_{i=1}^s \cZ_i$ and $Z=\bigcup_{C\in\cZ} C$, $\sM$ matches each vertex $v_{i,\ell}$ ($i\in[s]$, $\ell\in[4]$) with some vertex in $Z$.  Let $E_s$ be the event that a tuple of $4s$ distinct vertices $v_{i,\ell}$ with the above properties exists. We will show that it is very unlikely that $E_s$ holds, given a random $k$-admissible $r$-tuple $\sM$ of perfect matchings. 
Given $1\le s\le A\varepsilon\nn^2$, let $M_s$ count the number of ways to choose $4s$ distinct vertices $v_{i,\ell}$ ($i\in[s]$, $\ell\in[4]$) so that, for each $i\in[s]$, the cells $C_{i,2},C_{i,3},C_{i,4}$ belong to $\cZ_i$. Also, define $M_0=1$ for convenience.
We will bound $M_s$ from above by $M_{\lfloor s/2\rfloor}$ times the number of choices for the remaining vertices $v_{\lfloor s/2\rfloor+1,\ell},\ldots,v_{s,\ell}$ ($\ell\in[4]$). Note that, if $i\ge \lfloor s/2\rfloor+1$, for each choice of $C_{i,1}$, there are $(2d_i+1)^2 \le 9d_i^2 \le 9(d_{\lfloor s/2\rfloor+1})^2$ choices for each $C_{i,\ell}$ ($\ell\in\{2,3,4\}$) (since $d_i\ge1$ for all $i\in[s]$). Moreover, each cell $C\in\cT(n,t)$ has at most $4t^2$ vertices. Therefore,
\begin{align*}
M_s &\le
M_{\lfloor s/2\rfloor} \binom{\nn^2}{\lceil s/2\rceil} \left(  9(d_{\lfloor s/2\rfloor+1})^2 \right)^{3\lceil s/2\rceil} (4t^2)^{4\lceil s/2\rceil}
\notag\\
&\le
M_{\lfloor s/2\rfloor} \left(\frac{e \nn^2}{\lceil s/2\rceil}\right)^{\lceil s/2\rceil} \left( \frac{9A^2}{\log^2(1/\varepsilon)}\log^2\left(\frac{\nn^2}{\lfloor s/2\rfloor+1}\right) \right)^{3\lceil s/2\rceil} (4t^2)^{4\lceil s/2\rceil}
\notag\\
&= M_{\lfloor s/2\rfloor} \left(
2^83^6A^6 e \frac{t^8}{\log^6(1/\varepsilon)} \frac{\nn^2}{\lceil s/2\rceil} \log^6\left(\frac{\nn^2}{\lfloor s/2\rfloor+1}\right)  \right)^{\lceil s/2\rceil}.
\end{align*}
This combined with an easy inductive argument implies that, for every $1\le s\le A\varepsilon\nn^2$,
$$
M_s \le \left(
10^7A^6 \frac{t^8}{\log^6(1/\varepsilon)} (\nn^2/s) \log^6\left(\nn^2/s\right)  \right)^s .
$$
Now observe that, regardless of the choice of the $4s$ vertices $v_{i,\ell}$,
\begin{align}
|\cZ|  \le \sum_{i=1}^s |\cZ_i| &\le \sum_{i=1}^s 9{d_i}^2  = \sum_{i=1}^s \frac{9A^2}{\log^2(1/\varepsilon)}\log^2(\nn^2/i)
\le \frac{9A^2}{\log^2(1/\varepsilon)} \left( \sum_{i=1}^s \log(\nn^2/i) \right)^2
\notag\\
&= \frac{9A^2}{\log^2(1/\varepsilon)} \log^2(\nn^{2s}/s!)
\le \frac{9A^2}{\log^2(1/\varepsilon)} s \log^2(e\nn^2/s)
\le \frac{10A^2}{\log^2(1/\varepsilon)} s \log^2(\nn^2/s).
\label{eq:boundcalZ}
\end{align}
We will use Lemma~\ref{lem:switchings} to bound the probability $P_s$ that each vertex in $S=\{v_{i,\ell} : i\in[s], \ell\in[4]\}$ is matched by a random $k$-admissible perfect matching of $\sM$ to a vertex in $Z=\bigcup_{C\in\cZ} C$. Let $z=|Z|$, and recall $|S|=4s$ with $s\le A\eps \nn^2$. Then, from~\eqref{eq:boundcalZ} and the fact that each cell has at most $4t^2$ vertices, we get
\begin{equation}\label{eq:boundz}
z \le 4t^2|\cZ| \le \frac{40A^3\eps t^2}{\log^2(1/\varepsilon)} \lfloor n/t\rfloor^2 \log^2(1/(A\eps)) \le  40A^3 \eps n^2.
\end{equation}
Our assumptions in~\eqref{eq:kret} imply $r\le k \le (1/\eps)^{1/4}$. Using this fact and~\eqref{eq:boundz}, yields
\[
4erz \le 160eA^3 \eps^{3/4} n^2 \le n^2/2
\]
and also
\[
z+2 (4k+r+2)^2(4rs) \le z+ 400k^3 s \le 40A^3 \eps n^2+ 400(1/\eps)^{3/4} A\eps \nn^2 \le n^2/2,
\]
which are the two conditions we need to apply Lemma~\ref{lem:switchings}. Hence, by Lemma~\ref{lem:switchings} and using~\eqref{eq:boundcalZ} and the first step in~\eqref{eq:boundz},
\[
P_s = (16rz/n^2)^{2s} \le  (64rt^2|\cZ|/n^2)^{2s}
\le \left( \frac{640 A^2 r}{\log^2(1/\varepsilon)} (s/\nn^2) \log^2(\nn^2/s) \right)^{2s}.
\]
We conclude that, for $1\le s\le A\varepsilon\nn^2$,
\[
\pr(E_s) \le M_sP_s
\le \left( 10^{13} A^{10} \frac{r^2t^8}{\log^{10}(1/\varepsilon)} (s/\nn^2) \log^{10}\left(\nn^2/s\right)  \right)^s
\le \left( 10^{13} A^{11}  r^2 t^8 \varepsilon \right)^s
\le \varepsilon^{s/2},
\]
where we used~\eqref{eq:kret} and the fact that $\varepsilon_0$ is sufficiently small.
Summing over $s$, since the ratio $\pr(E_{s+1})/\pr(E_{s}) \le \eps^{1/2}<1/2$ and using~\eqref{eq:kret} once again,
\[
\sum_{s=1}^{\lfloor A\varepsilon\nn^2 \rfloor} \pr(E_s)
\le 2 \pr(E_1)
= O\left(\frac{r^2t^8\log^{10}\nn}{\nn^2}\right) = O\left(\frac{r^2t^{10}\log^{10}n}{n^2}\right) = o(1).
\]
\end{proof}
We have all the ingredients we need to prove our main result.
\begin{proof}[Proof of Theorem~\ref{thm:main}]
Pick a sufficiently small constant $p_0>0$, and suppose $p$, $k$ and $r$ satisfy~\eqref{eq:pkrthm}. Define $t$ and $\varepsilon$ as in~\eqref{eq:kte}, so the conclusion of Proposition~\ref{prop:phase1} is true for the $2r$-majority model (note that $2r\le pk/9$).
Moreover, let $\varepsilon_0={p_0}^{100}$, and assume that $\varepsilon_0$ is small enough as required by Lemma~\ref{lem:nostable}. We have $\eps \le (\frac{1000}{p}\log(1/p))^{-100} \le p^{100} \le \eps_0$.
Note that our choice of $k$, $r$, $\eps$ and $t$ trivially satisfies~\eqref{eq:kret}.

Let $U\subseteq[n]^2$ be the initial set of inactive vertices, and let $U^\circ$ be the $(2k+2)$-core $U^\circ$ of the subgraph of $\sL^*(n,k,r)$ induced by $U$ (i.e.~the final set of inactive vertices of $\M_r(\sL^*(n,k,r);p)$). Let $\cR$ be the set of cells in $\cT(n,t)\simeq[ \lfloor n/t\rfloor ]^2$ that contain some vertex in $U^\circ$. Since $U^\circ$ is contained in the $(2k-r+2)$-core of the subgraph of $\sL(n,k)$ induced by $U$ (i.e.~the final set of inactive vertices of $\M_{2r}(\sL(n,k);p)$), Proposition~\ref{prop:phase1} shows that a.a.s.\ the set of cells $[\lfloor n/t\rfloor]^2\setminus\cR$ contains an $\eps$-ubiquitous $\ell_1$-component. Therefore, the $\ell_\infty$-components of $\cR$, namely $\cB_1,\ldots,\cB_s$, must satisfy properties~(iii) and~(iv) in the definition of $\eps$-ubiquitous and, by Lemma~\ref{lem:needstable}, must be a stable collection of sets of cells with respect to a random $r$-tuple $\sM$ of $k$-admissible perfect matchings of $[n]^2$. Finally, Lemma~\ref{lem:nostable} claims that a.a.s.\ there are no such stable collections, and therefore $U$ must be empty. This concludes the proof of the theorem.
\end{proof}

\end{document}